\documentclass[11pt]{article}
\usepackage{graphicx} % Required for inserting images
\usepackage{amssymb, amsmath, amsthm, graphicx,mathrsfs}
\usepackage{caption,cite}
\usepackage{mathtools,color}
\usepackage[linktocpage=true]{hyperref}
\usepackage{comment}
\usepackage[margin=2.5cm]{geometry}

\newtheorem{thm}{Theorem}[section]
\newtheorem{cor}[thm]{Corollary}

\newtheorem{lem}[thm]{Lemma}

\newtheorem{prop}[thm]{Proposition}
\newtheorem{claim}[thm]{Claim}
\newtheorem{fact}[thm]{Fact}

\newcommand{\hh}{\mathcal{H}}
\newcommand{\hb}{\mathcal{B}}
\newcommand{\hd}{\mathcal{D}}
\newcommand{\hk}{\mathcal{K}}
\newcommand{\hf}{\mathcal{F}}
\newcommand{\hg}{\mathcal{G}}
\newcommand{\he}{\mathcal{E}}
\renewcommand{\l}{\left}
\renewcommand{\r}{\right}

\title{On the Matching Problem in Random Hypergraphs}

\author{Peter Frankl\footnote{R\'{e}nyi Institute, Budapest, Hungary. Email: \texttt{frankl.peter@renyi.hu}.}\quad\quad
Jiaxi Nie\footnote{School of Mathematics, Georgia Institute of Technology, Atlanta, GA 30332, USA. Email: \texttt{jnie47@gatech.edu}.}\quad\quad
Jian Wang\footnote{Department of Mathematics, Taiyuan University of Technology, Taiyuan, 030024, China. Research supported by National Natural Science Foundation of China No. 12471316. Email: \texttt{wangjian01@tyut.edu.cn}.}
}

\date{}
\begin{document}

\maketitle

\begin{abstract}
We study a variant of the Erd\H{o}s Matching Problem in random hypergraphs. Let $\hk_p(n,k)$ denote the Erd\H{o}s-Rényi random $k$-uniform hypergraph on $n$ vertices where each possible edge is included with probability $p$. We show that when $n\gg k^{2}s$ and $p$ is not too small, with high probability, the maximum number of edges in a sub-hypergraph of $\hk_p(n,k)$ with matching number $s$ is obtained by the trivial sub-hypergraphs, i.e. the sub-hypergraph consisting of all edges containing at least one vertex in a fixed set of $s$ vertices.
\end{abstract}

\section{Introduction}
In this paper, we consider the so-called Tur\'{a}n problem concerning matchings in random hypergraphs. Let $\hk(n,k)$ denote the complete $k$-graph on $n$ vertices--often denoted as $\binom{[n]}{k}$. For a fixed $p$, $0<p<1$ let $\hk_p=\hk_p(n,k)$ denote the {\it random} $k$-graph that we obtain by choosing each edge of $\binom{[n]}{k}$ {\it independently} and with probability $p$. By the law of large numbers with (very) high probability $|\hk_p|/\binom{n}{k}$ is very close to $p$.

Similarly, for every $x\in [n]$, the full star $\hk_p[x]=\{K\in \hk_p\colon x\in K\}$ has size very close to $p\binom{n-1}{k-1}$. E.g., by Chernoff's inequality
\begin{align}\label{ineq-chernoff}
Pr\left(\left|\frac{|\hk_p[x]|}{p\binom{n-1}{k-1}}-1\right|>\varepsilon\right)<2e^{-\frac{\varepsilon^2}{3}p\binom{n-1}{k-1}}.
\end{align}

The central problem that we consider is the following. Let $s$ be a fixed positive integer. Try and determine the size and structure of largest subhypergraphs of $\hk_p$ which do not contain $s+1$ pairwise disjoint edges.

To put this problem into context let us first recall the corresponding problem for the complete $k$-graph, $\binom{[n]}{k}$, that is, the case $p=1$.

\vspace{5pt}
{\bf\noindent Erd\H{o}s Matching Conjecture \cite{erdos1965problem}.} Suppose that $n,k,s$ are positive integers, $n\geq k(s+1)$ and $\he\subset \binom{[n]}{k}$ contains no $s+1$ pairwise disjoint edges. Then
\begin{align}\label{ineq-emc}
|\he| \leq \max\left\{\binom{k(s+1)-1}{k}, \binom{n}{k}-\binom{n-s}{k}\right\}.
\end{align}
\vspace{5pt}

Note that the simple constructions giving rise to the formulae on the right hand side are
\[
\binom{[k(s+1)-1]}{k} \mbox{ and }\he_T=\left\{E\in \binom{[n]}{k}\colon E\cap T\neq \emptyset\right\},
\]
where $T$ is some fixed $s$-subset of $[n]$. There has been a lot of research going on concerning this conjecture. Let us recall some of the results.

The current best bounds establish \eqref{ineq-emc} for $n>(2s+1)k-s$ ~\cite{frankl2013improved} and for $s>s_0$, $n>\frac{5}{3}sk$ ~\cite{frankl2022erdHos}.

The $s=1$ case has received the most attention. A family $\hf\subset \binom{[n]}{k}$ is called {\it intersecting} if $\hf$ contains no two pairwise disjoint members (edges). Say that $\hf$ is a \emph{star} if there is a vertex that lies in all edges.

\begin{thm}[Erd\H{o}s-Ko-Rado Theorem, \cite{erdos1961intersection}]
Suppose that $\hf\subset \binom{[n]}{k}$ is intersecting and $n\geq 2k$. Then
\[
|\hf| \leq \binom{n-1}{k-1}.
\]
Further, when $n\ge 2k+1$, the equality holds only for a star.
\end{thm}
Concerning the probabilistic versions of EKR, Balogh, Bohman and Mubayi\cite{balogh2009erdHos} have proved several strong results. Here we mention some of them. Let $f(n)$, $g(n)$ be functions of $n$. The notation $f(n)\ll g(n)$ means $f(n)/g(n)\rightarrow 0$ as $n\rightarrow \infty$.
\begin{thm}[\cite{balogh2009erdHos}]
When $k\ll n^{1/4}$ and $p\in [0,1]$, w.h.p. the maximum sized intersecting subhypergraph in $\hk_p(n,k)$ is a star. Furthermore, the same conclusion also holds when $n^{1/4}\ll k\ll n^{1/3}$, and $p\gg n^{-1/4}/\binom{n-1}{k-1}$ or $p\ll k^{-1}/\binom{n-1}{k-1}$.
\end{thm}
For more on the probabilistic EKR, see also~\cite{hamm2019erdHosI,hamm2019erdHos,gauy2017erdHos,balogh2023sharp}.

Let us recall the notation $\nu(\hf)$ for the {\it matching number}, that is, $\nu(\hf)$ is the maximum number of pairwise disjoint edges in $\hf$. In particular, $\hf$ is intersecting if and only if $\nu(\hf)\leq 1$. 
\begin{comment}
The following old result is a generalization of EKR.

\begin{thm}[\cite{frankl1987shifting}]
\begin{align}\label{ineq-emcub}
|\hf| \leq \nu(\hf) \binom{n-1}{k-1} \mbox{ for all } \hf\subset \binom{[n]}{k}.
\end{align}
\end{thm}
\end{comment}

%Let $\Delta\hh$ denote $\max_x |\hh[x]|$, the {\it maximum degree} of $\hh$.

One natural question is to investigate for which range of the values $n,k,s,p$ does
\begin{align}\label{ineq-1.4}
|\hf| \leq \max\{|\hk_p(\hat{S})|\colon S\in\binom{[n]}{s}\}
\end{align}
hold with high probability, where $\nu(\hf)=s$ and $\hk_p(\hat{S}):=\{E\in \hk_p\colon E\cap S\neq \emptyset\}$. The fact that $|\hk_p(\hat{S})|$ is asymptotically $p\left(\binom{n}{k}-\binom{n-s}{k}\right)$ when $p$ is not too small follows from the law of  large numbers.

Let $\tau(\hf)$ be the \emph{covering number} of $\hf$, which is the minimum size of a set $W$ of vertices  such that $E\cap W\not=\emptyset$ for each $E\in \hf$. Clearly, $\nu(\hf)\le \tau(\hf)$ always. We say that $\hf$ is \emph{trivial} if $\nu(\hf)=\tau(\hf)$. Otherwise $\hf$ is called {\it non-trivial}. We prove the following result.
\begin{thm}\label{thm:main}
For any integers $s=s(n)$, $k=k(n)$ and real number $p=p(n)$ with
\begin{itemize}
    \item[(i)] $\frac{64ks\log n}{\binom{n-1}{k-1}}\le p\le 1$, and
    \item[(ii)] $n\geq 200k^3s$,
\end{itemize}
with high probability (w.h.p.), every non-trivial $\hf\subseteq \hk_p(n,k)$ with $\nu(\hf)\leq s$ satisfies
$$
|\hf|< |\hk_p(\hat{S})|,
$$
for some $S\in \binom{[n]}{s}$ .
\end{thm}

We also prove that the lower bound for $n$ can be improved to $O(k^2s)$ at the expense of a worse lower bound for $p$.

\begin{thm}\label{thm:main-2} 
Let $t$  be an integer with $2\leq t\ll k$.  For any integers $s=s(n)$, $k=k(n)$ and real number $p=p(n)$ with
\begin{itemize}
    \item[(i)] $\frac{8(t+1)k^{t+1}s^t\log n}{\binom{n-1}{k-1}}\le p\le 1$, and
    \item[(ii)] $56k^{2+1/t}s \leq n\leq 56k^{2+1/(t-1)}s$,
\end{itemize}
with high probability, every non-trivial $\hf\subseteq \hk_p(n,k)$ with $\nu(\hf)\leq s$ satisfies
$$
|\hf|< |\hk_p(\hat{S})|,
$$
for some $S\in \binom{[n]}{s}$.
\end{thm}

Let $0<\epsilon<1$ and  $t=k^{\epsilon}$ in Theorem \ref{thm:main-2}. Let $f(x):=\frac{\log x}{x^\epsilon}$. Then
\[
f'(x) = \frac{1}{x^{1+\epsilon}} -\epsilon\frac{\log x}{x^{1+\epsilon}} = \frac{1-\epsilon \log x}{x^{1+\epsilon}}. 
\]
Clearly $x^{1+\epsilon}>0$ for $x\geq 2$ and $1-\epsilon \log x\geq 0$ for $x\leq e^{1/\epsilon}$,  $1-\epsilon \log x\leq 0$ for $x\geq e^{1/\epsilon}$. It follows that
\[
f(x) \leq f\left(e^{1/\epsilon}\right) = \frac{1}{e\epsilon}. 
\]
Thus $k^{1/t}= k^{1/k^\epsilon}=e^{f(k)}<e^{{1}/{(e\epsilon)}}$ and we have the following corollary.

\begin{cor}\label{cor:main}
  For $0<\epsilon<1$, $n\geq 56 e^{1/(e\epsilon)} k^2s$  and $p\geq \frac{10(ks)^{k^\epsilon}k^{1+\epsilon}}{\binom{n-1}{k-1}}$, with high probability every non-trivial $\hf\subseteq \hk_p(n,k)$ with $\nu(\hf)=s$ satisfies
$|\hf|< |\hk_p(\hat{S})|$ for arbitrary $S\in \binom{[n]}{s}$.  
\end{cor}

Our next result shows that Theorem \ref{thm:main-2} does not hold for the range $p> \frac{8(t+1)k^{t+1}s^t\log n}{\binom{n-1}{k-1}}$ when $k^2s\geq 2(t+3)n\log n$.

\begin{prop}\label{prop-2.3}
If
\[
\frac{s\log n}{\binom{n-s}{k}} \ll p\ll \frac{e^{\frac{k^2s}{2n}}}{\binom{n}{k}},
\]
then with high probability, $\hk_p(n,k)$ has matching number at most $s$ and is non-trivial.
\end{prop}

Note that for $k^2 s\geq  2(t+3)n\log n$ and $n>ks$ we have
\[
 \frac{k}{n}  e^{\frac{k^2s}{2n}} \geq \frac{k}{n} e^{(t+3)\log n}= kn^{t+2} > k(ks)^{t+2} =(k^{t+1}s^t) k^2s^2\gg (k^{t+1}s^t)8(t+1)\log n.
\]
It follows that
\[
\frac{e^{\frac{k^2s}{2n}}}{\binom{n}{k}} = \frac{e^{\frac{k^2s}{2n}}}{\binom{n-1}{k-1}} \frac{k}{n} \gg \frac{8(t+1)k^{t+1}s^t\log n}{\binom{n-1}{k-1}}.
\]
%\JN{How do we get the last inequality?}
Thus Proposition \ref{prop-2.3} implies that  Theorem \ref{thm:main-2} does not hold for the range $n\log n\ll k^2 s$ and $\frac{8(t+1)k^{t+1}s^t\log n}{\binom{n-1}{k-1}}\ll p\ll \frac{e^{\frac{k^2s}{2n}}}{\binom{n}{k}}$ with some constant integer $t<k$.

% Note also that if $2n\log n=k^{2-\epsilon}s$ for some $0<\epsilon<1$ then $2k^{\epsilon}n\log n=k^2s$. It follows that 
% \[
%  \frac{e^{\frac{k^2s}{2n}}}{\binom{n}{k}} > \frac{n^{k^{\epsilon}}k^k}{(en)^k} =\frac{k^k}{e^kn^{k-k^{\epsilon}}}.
% \]
% Let $f(x)=k^x$. Then $f'(x)=k^x\log k$, and hence $f(1)-f(\epsilon)=k-k^{\epsilon} = (1-\epsilon) k^{\xi}\log k< k^{\xi'}$ for some $\epsilon<\xi<\xi'<1$. For $k$ sufficiently large with respect to $s$, by $n<k^2s$ we have
% \[
%  \frac{e^{\frac{k^2s}{2n}}}{\binom{n}{k}} >\frac{k^k}{e^k n^{ k^{\xi'}}} >\frac{k^{k-2k^{\xi'}}}{e^k s^{k^{\xi'}}}>1.
% \]
% Thus Corollary \ref{cor:main} does not hold for $n=Ck^{2-\epsilon'}s$ for any $\epsilon'>0$. 

For $k=2$, we obtain the following result. We write $a=(1\pm \varepsilon)b$ for $(1-\varepsilon)b\leq a\leq (1+\varepsilon)b$.

\begin{thm}\label{thm:main-3}
Let $0<\varepsilon<1$ and $n\geq 2s+2$. Let $X$ be the maximum number of edges in a subgraph $F$ of $G(n,p)$ with $\nu(F)\leq s$. If $p\geq \frac{250\log n}{\varepsilon^2 n}$, then with probability   at least $1-2n^{-s}$,
\[
X= (1\pm\varepsilon)p\max\left\{\binom{2s+1}{2},\binom{s}{2}+s(n-s)\right\}.
\]
\end{thm}

Let $\hf\subset \binom{[n]}{k}$. 
For $R,Q\subset [n]$ with $R\cap Q=\emptyset$, define
\[
\hf(Q)=\{F\setminus Q\colon Q\subset F\in \hf\},\ \hf(\bar{R},\hat{Q})=\{F\in \hf\colon F\cap R=\emptyset, F\cap Q\neq \emptyset\}.
\]
For $R=\emptyset$  we write $\hf(\bar{R},\hat{Q})$ as $\hf(\hat{Q})$. For $Q=[n]\setminus R$ we  write $\hf(\bar{R},\hat{Q})$ as $\hf(\bar{R})$. 

\section{A weaker result and the outline of the main proof}
Let $\binom{[n]}{\leq k}$ denote the collection of all subsets of $[n]$ of size at most $k$. For  $\hg \subset \binom{[n]}{\leq k}$, let 
\[
\langle \hg \rangle =\left\{F\in \binom{[n]}{k}\colon \mbox{ there exists }G \in \hg \mbox{ such that }G\subset F\right\}.
\]

We need the following version of the Chernoff's bound.

\begin{thm}[\cite{janson2011random}]
Let $X\in Bi(n,p)$, $\mu=np$. If $0<\varepsilon \leq 3/2$ then
\begin{align}\label{chernoff-small}
Pr(|X-\mu|\geq \varepsilon \mu) \leq 2e^{-\frac{\varepsilon^2}{3}\mu}.
\end{align}
If $x\geq 7\mu$ then
\begin{align}\label{chernoff-large}
Pr(X\geq x) \leq e^{-x}.
\end{align}
\end{thm}

We say that a $k$-graph $\hh$ is \emph{resilient} if for any vertex $u\in V(\hh)$, $\nu(\hh-u)=\nu(\hh)$; see~\cite{frankl2018resilient} for the maximum size of a resilient hypergraph with given matching number.

Consider $\hf\subset \hk_p$ with $\nu(\hf) \leq s$ and of maximal size. The ultimate goal is to prove
\begin{align}\label{ineq-new2-1}
|\hf| \leq \max\left\{|\hk_p(\hat{S})|\colon |S|=s\right\}.
\end{align}

For the complete graph, i.e., for $p=1$,
\[
|\hk_1(\hat{S})| =\binom{n}{k}-\binom{n-s}{k}, \mbox{ for all $s$-sets }S.
\]
For $p<1$ the {\it expected } size of $\hk_p(\hat{S})$ is $p\left(\binom{n}{k}-\binom{n-s}{k}\right)$ but it is impossible to know its {\it exact} value. 

How to prove \eqref{ineq-new2-1} without this piece of information? Since \eqref{ineq-new2-1} is evident for $\hf=\hk_p(\hat{S_0})$ with $S_0$ an $s$-element set, to start with we may assume that $\hf$ is not of this form. That is, $\tau(\hf)>\nu(\hf)$, meaning that $\hf$ is non-trivial. 

With this assumption we want to prove a {\it stability} result, i.e., a bound considerably smaller than $p\left(\binom{n}{k}-\binom{n-s}{k}\right)$. For $p=1$ this was done by Bollob\'{a}s, Daykin and Erd\H{o}s \cite{BDE} who generalized the Hilton-Milner Theorem ($\nu=1$ case) to this situation. Unfortunately, their argument does not seem to work for the $p<1$ case.

Here is what we can do. Choose a set $R$ of maximal size satisfying 
\begin{align}\label{ineq-new2-2}
\nu(\hf(\bar{R}))=s-|R|.
\end{align}
Set $q=s-|R|$. By our assumption $1\leq q\le s$. Moreover, for all elements $x\in [n]\setminus R$, by the maximal choice of $R$, $\nu(\hf(\overline{R\cup \{x\}}))=q$. That is, $\hf(\bar{R})$ is resilient. Now it suffices to show that $|\hf(\bar{R})|$ is smaller than $\frac{1}{2}pq\binom{n-1}{k-1}$, which is a lower bound for $\max\left\{|\hk_p(\hat{S})|\colon |S|=q\right\}$.

\begin{lem}\label{lem-new2-0}
Suppose that $\hf$ is a resilient $k$-graph, $\nu(\hf) =q$. Then there is a collection $\hb$ of 2-element sets satisfying $\hf\subset \langle \hb\rangle$ and $|\hb|\leq (kq)^2$.
\end{lem}
\begin{proof}
Let $G_1,\ldots,G_q\in \hf$ form a maximal matching and set $Y=G_1\cup \ldots\cup G_q$. Then $F\cap Y\neq \emptyset$ for all $F\in \hf$. For each $y\in Y$, $\nu(\hf(\bar{y}))=q$ by resilience. Define $Z_y=H_1^{(y)}\cup \ldots\cup H_q^{(y)}$ where $H_1^{(y)},\ldots,H_q^{(y)}$ form a maximal matching in $\hf(\bar{y})$. Now $\hb=\{(y,z_y)\colon y\in Y,z_y\in Z_y\}$ satisfies the requirements.
\end{proof}

For the case $n\gg k^3q$ and $p\gg \frac{k^2q \log n}{\binom{n-1}{k-1}}$ Lemma \ref{lem-new2-0} implies 
\begin{align}
|\hf(\bar{R})| \leq \frac{1}{2} pq\binom{n-1}{k-1} \mbox{ with high probability.}
\end{align}
Indeed, since $|\hb|\le (kq)^2$ we only need to show 
\begin{align}\label{ineq-new2-3}
|\hf(\{x_1,x_2\})|\le |\hk_p(\{x_1,x_2\})| \le \frac{p}{2k^2q}\binom{n-1}{k-1}
\end{align}
for any 2-subset $\{x_1,x_2\}\subset [n]\setminus R$. Note that, since $n\gg k^3q$, the expected size of $|\hk_p(\{x_1,x_2\})|$ is 
$$p\binom{n-2}{k-2}\le \frac{pk}{n}\binom{n-1}{k-1}\ll \frac{p}{k^2q}\binom{n-1}{k-1}.$$
Thus (\ref{ineq-new2-3}) is true by combining the Chernoff bound, the union bound and $p\gg \frac{k^2q\log n}{\binom{n-1}{k-1}}$. (there are ``only" $\binom{n-r}{2}$ choices for $\{x_1,x_2\}$).

How to improve on the bound for $p$? Note that we do not need \eqref{ineq-new2-3} for each individual pair but only an upper bound for $|\hf\cap\langle\hb\rangle|$. Let us use the special structure of $\hb$. By our proof of Lemma~\ref{lem-new2-0}, $\hb$ is the (not necessarily disjoint) union of $qk$ ``fans" $\{\{y,z_y\}\colon z_y\in Z\}=:\hd_{y,Z}$. In fact, by a more careful analysis we can reduce the number of fans to $q$ (at the expense of adding some cliques, which is even easier to control, see Section~\ref{sec:main}). Thus, instead of \eqref{ineq-new2-3} it is sufficient  to require 
\begin{align}\label{ineq-new2-4}
|\hf\cap \langle \hd_{y,Z} \rangle| \leq \frac{1}{2}p\binom{n-1}{k-1} \mbox{ for each fan } \hd_{y,Z}.
\end{align}
Here we'll have $n\times \binom{n-1}{qk}$ requirements--much more than $\binom{n}{2}$ but the probability that \eqref{ineq-new2-4} is violated for any fixed fan is going down exponentially and we should get bounds better than requiring \eqref{ineq-new2-3}. 
% Since we have  no control over $q$, we might as well require \eqref{ineq-new2-4} for the case $q=1$, i.e., fan of size $k$.
We will elaborate on this idea in Section~\ref{sec:main}, which eventually gives Theorem~\ref{thm:main}.

\section{Proof of Theorem \ref{thm:main}}\label{sec:main}

\begin{lem}\label{lem-key}
Let $\hk_p=\hk_p(n,k)$. For $k,s\ge2$, $n\geq 200k^3s$, $p\geq \frac{64ks\log n}{\binom{n-1}{k-1}}$, properties \textit{(i)}, \textit{(ii)} and \textit{(iii)} hold for every $1\le q\le s$ with high probability.
\begin{itemize}
    \item[(i)] For $R,Q\subset [n]$  with $R\cap  Q=\emptyset$,  $|R|=s-q$ and $|Q|=q$,
    \begin{align}\label{ineq-2-1}
    |\hk_p(\bar{R},\hat{Q})|> \frac{1}{2}pq\binom{n-1}{k-1}.
    \end{align}
    \item[(ii)] For every $Q\subset[n]$  with  $|Q|< 3kq\leq 3ks$,
    \begin{align}\label{ineq-2-2}
    \left|\hk_p\cap \left\langle \binom{Q}{2}\right\rangle\right|< \frac{1}{4}p q\binom{n-1}{k-1}.
    \end{align}
    \item[(iii)] For every $x\in [n]$, $Q\subset [n]\setminus\{x\}$ with  $|Q|= kq\leq ks$,
    \begin{align}\label{ineq-2-3}
    |\hk_p\cap \langle\hd_{x,Q}\rangle|< \frac{1}{4}p\binom{n-1}{k-1}.
    \end{align}
\end{itemize}
\end{lem}
\begin{proof}
\begin{itemize}
    \item[\textit{(i)}] Let $\hk=\binom{[n]}{k}$. It is easy to see that
\[
|\hk(\bar{R},\hat{Q})| =\binom{n-|R|}{k}-\binom{n-|R|-|Q|}{k}\geq q\binom{n-s}{k-1} .
\]
Note that
$$
\frac{\binom{n-s}{k-1}}{\binom{n-1}{k-1}}\ge \l(\frac{n-k-s}{n-k}\r)^{k-1}\ge 1-\frac{s(k-1)}{n-k}\ge \frac{3}{4}.
$$

Thus we have
$$
|\hk(\bar{R},\hat{Q})|\ge\frac{3}{4}q\binom{n-1}{k-1}.
$$

\begin{comment}
By \eqref{chernoff-small}, we infer that
\[
Pr\left(|\hk_p(\bar{T},R)|\leq \frac{1}{2}pr\binom{n-1}{k-1}\right)\leq  e^{-\frac{1}{9}pr\binom{n-1}{k-1}}.
\]
\end{comment}
Applying \eqref{chernoff-small} with $\varepsilon=1/4$, we infer that
\[
Pr\left(|\hk_p(\bar{R},\hat{Q})|\leq \frac{1}{2}pq\binom{n-1}{k-1}\right)\leq  2e^{-\frac{1}{64}pq\binom{n-1}{k-1}}\le 2e^{-\frac{1}{64}p\binom{n-1}{k-1}}.
\]
Since for all $1\le q\le s$ the number of such pairs $(R,Q)$ is at most $n^{s}2^{s}$, by the union bound and  $p\ge \frac{64ks \log n}{\binom{n-1}{k-1}}$, the probability that \textit{(i)} does not hold for some $1\le q\le s$ is at most
$$
n^s2^s2e^{-\frac{1}{64}p\binom{n-1}{k-1}}\le n^s2^{s+1}n^{-ks}\le 2^{s+1}n^{-s}\rightarrow 0~\text{as}~n\rightarrow\infty.
$$

\item[\textit{(ii)}] It suffices to show that with high probability (\ref{ineq-2-2}) holds for every $1\le q\le s$ and $Q\subset[n]$  with  $|Q|= 3kq$. By $n\geq 200k^3s\geq 200 k^3q$,
\[
 \left|\left\langle \binom{Q}{2}\right\rangle\right|\leq \binom{|Q|}{2} \binom{n-2}{k-2}\leq \frac{9k^2q^2}{2} \frac{k}{n} \binom{n-1}{k-1}\leq \frac{q}{28}\binom{n-1}{k-1}.
\]
By \eqref{chernoff-large},
\[
Pr\left(\left|\hk_p\cap \left\langle \binom{Q}{2}\right\rangle\right|\geq \frac{1}{4}p q\binom{n-1}{k-1}\right) \leq e^{-\frac{1}{4}pq\binom{n-1}{k-1}}.
\]
Since for fix $q$ the number of such $Q$ is at most $\binom{n}{3kq}$ and $p\geq  \frac{64ks\log n}{\binom{n-1}{k-1}}$, by the union bound, the probability that \eqref{ineq-2-2} doesn't hold for some $1\le q\le s$ is at most
$$
\sum_{q=1}^{s}\binom{n}{3kq}e^{-\frac{1}{4}pq\binom{n-1}{k-1}}\le \sum_{q=1}^sn^{3kq-16ksq}\le \sum_{q=1}^s n^{-16kq}\le sn^{-16k}\rightarrow 0~\text{as}~n\rightarrow\infty.
$$

\item[\textit{(iii)}] Finally for every $x\in [n]$, $Q\subset [n]\setminus\{x\}$ with  $|Q|= kq\leq ks$, by $n\geq 200k^3s$,
\[
|\langle\hd_{x,Q}\rangle| \leq kq \binom{n-2}{k-2} \leq kq \frac{k}{n} \binom{n-1}{k-1} \leq \frac{1}{28} \binom{n-1}{k-1}.
\]
By \eqref{chernoff-large},
\[
Pr\left(|\hk_p\cap \langle\hd_{x,Q}\rangle|\geq \frac{1}{4}p \binom{n-1}{k-1}\right) \leq e^{-\frac{1}{4}p \binom{n-1}{k-1}}.
\]
Since for fix $1\le q\le s$ the number of such pair $(x,Q)$ is at most $n\binom{n}{kq}$ and $p\geq  \frac{64ks\log n}{\binom{n-1}{k-1}}$, by the union bound, the probability that \eqref{ineq-2-3} doesn't hold for some $1\le q\le s$ is at most
$$
\sum_{q=1}^s n\binom{n}{kq}e^{-\frac{1}{4}p \binom{n-1}{k-1}}\le \sum_{q=1}^sn^{1+kq-16ks}\le sn^{-10ks}\rightarrow 0~\text{as}~n\rightarrow\infty.
$$
\end{itemize}
\end{proof}

\begin{proof}[Proof of Theorem \ref{thm:main}]
By Lemma \ref{lem-key}, with high probability, $\hk_p$ satisfies \textit{(i)}, \textit{(ii)} and \textit{(iii)} of Lemma \ref{lem-key} for every $1\le q\le s$. Let us fix a hypergraph $\hf$  that satisfies \textit{(i)}, \textit{(ii)} and \textit{(iii)} for every $1\le q\le s$ and let $\hh\subset \hf$ be a sub-hypergraph with $\nu(\hh)=s$ and  $|\hh|$ maximal. We have to show that $\hh$ is trivial.
Arguing indirectly assume that $\hh$ is not trivial. Hence there exists $R$, $|R|=r$, $0\leq r<s$ such that $\hh(\bar{R})$ is resilient with $\nu(\hh(\bar{R}))=s-r$. 

Let $q=s-r$. Note that for an arbitrary $Q\in \binom{[n]\setminus R}{q}$ the family $\hf(\bar{R},\hat{Q})$ can be used to replace $\hh(\bar{R})$ and the resulting family  has matching number $s$.
To get a contradiction we need to show
\begin{align}
|\hf(\bar{R},\hat{Q})|>|\hh(\bar{R})|.
\end{align}

By \eqref{ineq-2-1},
\[
|\hf(\bar{R},\hat{Q})|>\frac{q}{2}p\binom{n-1}{k-1}.
\]
Let $H_1,H_2,\ldots,H_q$ be a maximal matching in $\hh(\bar{R})$ and let $X=H_1\cup H_2\cup \ldots\cup H_q$. Define
\[
\hh_0=\{H\in \hh(\bar{R})\colon |H\cap X|\geq 2\}
\]
and
\[
\hh_i= \{H\in \hh(\bar{R})\colon  |H\cap X|=1, H\cap X= H\cap H_i\},\ i=1,2,\ldots,q.
\]
By \eqref{ineq-2-2},
\begin{align}\label{ineq-2-4}
|\hh_0| < \frac{1}{4}p q\binom{n-1}{k-1}.
\end{align}

We claim that $\hh_i$ is intersecting for $i=1,2,\ldots,q$. Indeed, otherwise if there exist $H_i',H_i''\in \hh_i$ such that $H_i'\cap H_i''=\emptyset$ then $H_1,\ldots,H_{i-1},H_i',H_i'',H_{i+1},\ldots,H_q$ form a matching of size $q+1$ in $\hh(\bar{R})$, a contradiction. Thus $\hh_i$ is intersecting. Without loss of generality, assume that $\hh_1,\ldots,\hh_\ell$  are stars and $\hh_{\ell+1},\ldots,\hh_q$ are non-trivial intersecting. Let $x_i$ be the center of the star $\hh_i$, $i=1,2,\ldots,\ell$. For each $i\in [\ell]$, since $\hh(\bar{R})$ is resilient, there exists a matching $L_1^i,L_2^i,\ldots,L_q^i$ in $\hh(\overline{R\cup\{x_i\}})$. Let 
\[
Q_i= L_1^i\cup L_2^i\cup \ldots\cup L_q^i.
\]
Clearly $x_i\in H$ and $H\cap Q_i\neq \emptyset$ for any $H\in \hh_i$ with $i\in [\ell]$. Thus by \eqref{ineq-2-3} we have
\begin{align}\label{ineq-2-5}
\sum_{1\leq i\leq \ell}|\hh_i|\leq \sum_{1\leq i\leq \ell} |\hf \cap \langle \hd_{x_i,Q_i}  \rangle|< \frac{1}{4}p\ell\binom{n-1}{k-1}.
\end{align}

Since $\hh_i$ is non-trivial intersecting for $\ell+1\leq i\leq q$, there exist $H_i',H_i''\in \hh_i$ such that $H_i'\cap H_i=\{x_i\}$ and $x_i\notin H_i''$. It follows that $|H\cap (H_i\cup H_i'\cup H_i'')|\geq 2$ for any $H\in \hh_i$. Let $Q=\cup_{\ell+1\leq i\leq q}(H_i\cup H_i'\cup H_i'')$. Clearly $|Q|< 3k(q-\ell)$. Then by \eqref{ineq-2-2} we infer that
\begin{align}\label{ineq-2-6}
\sum_{\ell+1\leq i\leq q} |\hh_i| < \left|\hf\cap \left\langle\binom{Q}{2}\right\rangle\right| \leq \frac{1}{4}p(q-\ell) \binom{n-1}{k-1}.
\end{align}
Adding \eqref{ineq-2-4}, \eqref{ineq-2-5} and \eqref{ineq-2-6}, we obtain that
\[
|\hh(\bar{R})| =\sum_{0\leq i\leq q} |\hh_i| < \frac{1}{2}pq\binom{n-1}{k-1}<|\hf(\bar{R},\hat{Q})|,
\]
the desired contradiction.
\end{proof}

\section{Proof of Theorem \ref{thm:main-2} and Proposition \ref{prop-2.3}}

In this section, we prove Theorem \ref{thm:main-2} and Proposition \ref{prop-2.3}. Note that Theorem \ref{thm:main-2} improves the lower bound for $n$ in Theorem~\ref{thm:main} at the expense of a worse lower bound for $p$.

We say that a $k$-graph $\hh$ is \emph{$t$-resilient} if $\nu(\hh-T)=\nu(\hh)$ for any $T\subset V(\hh)$ with $|T|\leq t$. Clearly $t\leq k-1$. Indeed, deleting $k$ vertices in an edge decreases the matching number by at least one.

Let us prove the following simple fact.

\begin{fact}\label{fact-3.1}
For every $\hh\subset \binom{[n]}{k}$ and $t\leq k-1$, there exist $\ell \leq  \nu(\hh)$ and $T_1,T_2,\ldots,T_\ell$  such that
\begin{itemize}
  \item[(i)] $|T_i|\leq t$ for all $i=1,2,\ldots,\ell$.
  \item[(ii)] For $i=1,2,\ldots,\ell$, $\hh-U_{i-1}$ is a $(|T_{i}|-1)$-resilient $k$-graph with matching number $\nu(\hh)-i+1$, where $U_{i-1}=T_1\cup T_2\cup \ldots\cup T_{i-1}$.
  \item[(iii)]  $\hh-U$ is a $t$-resilient $k$-graph with matching number $\nu(\hh)-\ell$, where $U=T_1\cup T_2\cup \ldots\cup T_{\ell}$.
\end{itemize}
\end{fact}
\begin{proof}
Let $\hh_0=\hh$. We obtain $T_1,T_2,\ldots,T_\ell$  by a greedy algorithm. For $i=0,1,\ldots$, repeat the following procedure. If  $\hh_i$ is $t$-resilient then let $\ell=i$ and stop. Otherwise find $T_{i+1}$ of the minimum size so that $\nu(\hh_i-T_{i+1})=\nu(\hh_i)-1$ and let $\hh_{i+1}=\hh_i-T_{i+1}$. Note that in each step the matching number of $\hh_i$ decreases by 1. The procedure will terminate in at most $\nu(\hh)$ steps.

Since $\hh_{i-1}$ is not $t$-resilient for $1\leq i\leq \ell$, we have $|T_i|\leq t$ and \textit{(i)} holds. Since $\hh_\ell = \hh-U$ is empty or $t$-resilient, \textit{(iii)} follows.

We are left with \textit{(ii)}. As $T_i$ is of the minimum size satisfying  $\nu(\hh_{i-1}-T_{i})=\nu(\hh_{i-1})-1$,  we infer that $\nu(\hh_{i-1}-R)=\nu(\hh_{i-1})$ for all $R\subset V(\hh_{i-1})$ with $|R|<|T_i|$. That is, $\hh_{i-1}=\hh-U_{i-1}$ is  $(|T_{i}|-1)$-resilient with matching number $\nu(\hh)-i+1$. Thus \textit{(ii)} holds.
\end{proof}

\begin{lem}\label{lem-3.1}
Let $\hh$ be a $t$-resilient $k$-graph with matching number $s$. Then there exists  $\hb\subset \binom{V(\hh)}{t+1}$ with $|\hb| \leq (ks)^{t+1}$ such that $\hh\subset \langle \hb\rangle$.
Moreover, for any $X\subset V(\hh)$ with $|X|\leq x<ks$,
there exists  $\hb'\subset \binom{V(\hh)}{t+1}$ with $|\hb'| \leq x(ks)^{t}$ such that $\hh(\hat{X}) \subset \langle \hb'\rangle$.
\end{lem}

\begin{proof}
Let us prove the lemma by a branching process of $t+1$ stages. A  sequence $S=(x_1,x_2,\ldots,x_\ell)$ is an ordered sequence of distinct elements of $V(\hh)$ and we use $\widehat{S}$ to denote the underlying unordered set $\{x_1,x_2,\ldots,x_\ell\}$.

At the first stage, let $H_1^1,\ldots,H_s^1$ be a maximal matching in $\hh$. Make $ks$ sequences $(x_1)$ with  $x_1\in H_1^1\cup \ldots\cup H_s^1$. At the $p$th stage for $p=1,2,3,\ldots, t$, for each  sequence $(x_1,\ldots,x_p)$ of length $p$,  since $\hh$ is $t$-resilient, then there exists a maximal matching $H_1^p,\ldots,H_s^p$ in $\hh-\{x_1,\ldots,x_p\}$.  We replace $(x_1,\ldots,x_p)$ by  $ks$  sequences $(x_1,\ldots,x_p,x_{p+1})$ with $x_{p+1}\in H_1^p\cup \ldots\cup H_s^p$. Eventually by the branching process we shall obtain at most $(ks)^{t+1}$ sequences of length $t+1$.

\begin{claim}\label{claim-3.1}
For each $H\in \hh$, there is a sequence $S$ of length $t+1$ with $\widehat{S}\subset H$.
\end{claim}

\begin{proof}
Let $S=(x_1,\ldots,x_\ell)$  be a sequence of maximal length that occurred
at some stage of the branching process satisfying $\widehat{S}\subset H$. Suppose indirectly that $\ell<t+1$.   Since  $H\cap (H_1^1\cup \ldots\cup H_s^1)\neq \emptyset$ at the first stage, there is a sequence $(x_1)$ with $x_1\in H_1^1\cup \ldots\cup H_s^1$ such that $\{x_1\}\subset H$. Thus $\ell\geq 1$. Since $\ell<t+1$, at the $\ell$th stage there is  a maximal matching $H_1^\ell,\ldots,H_s^\ell$ in $\hh-\widehat{S}$ and $S$ was replaced by $ks$ sequences $(x_1,\ldots,x_\ell,y)$ with $y\in H_1^\ell\cup \ldots\cup H_s^\ell$. Since $H\cap (H_1^\ell\cup \ldots\cup H_s^\ell)\neq \emptyset$, there is some $y_0\in H\cap (H_1^\ell\cup \ldots\cup H_s^\ell)$. Then $(x_1,\ldots,x_\ell,y_0)$ is a longer sequence satisfying $\{x_1,\ldots,x_\ell,y_0\}\subset H$, contradicting the maximality of $\ell$.
\end{proof}

Let $\hb$ be the collection of  $\widehat{S}$ over all sequences $S$ of length $t+1$. By Claim \ref{claim-3.1}, we infer that $\hh\subset \langle \hb\rangle$.

Similarly, if we make a branching process starting at $x$ sequences $(x_1)$ with  $x_1\in X$, then we shall obtain a $\hb'\subset \binom{V(\hh)}{t+1}$ with $|\hb'| \leq x(ks)^{t}$ such that $\hh(\hat{X}) \subset \langle \hb' \rangle$.
\end{proof}

\begin{lem}\label{lem-key2}
Let $\hk_p=\hk_p(n,k)$ and let $t$  be an integer with $1\leq t\ll k$.  If $n\geq 56k^{2+1/t}s$ and $p\geq \frac{8(t+1)k^{t+1}s^t\log n}{\binom{n-1}{k-1}}$, then with high probability \textit{(i)}, \textit{(ii)} and \textit{(iii)} hold.
\begin{itemize}
    \item[(i)] For $R,Q\subset [n]$  with $R\cap  Q=\emptyset$,  $|R|=s-q$ and $|Q|=q$,
    \begin{align}\label{ineq-3-1}
    |\hk_p(\bar{R},\hat{Q})|> \frac{1}{2}pq\binom{n-1}{k-1}.
    \end{align}
    \item[(ii)] For every $R\subset [n]$ with $2\leq |R|\leq t$,
    \[
    |\hk_p(R)|\leq \frac{1}{4|R|(ks)^{|R|-1}}p\binom{n-1}{k-1}.
    \]
    \item[(iii)] For every $T\subset [n]$ with $|T|= t+1$,
    \[
    |\hk_p(T)|\leq \frac{1}{4k^{t+1}s^t}p\binom{n-1}{k-1}.
    \]
\end{itemize}

\end{lem}
\begin{proof}
Clearly, \textit{(i)} follows from the proof of Lemma \ref{lem-key} \textit{(i)}.  Thus we only need to show \textit{(ii)} and \textit{(iii)}. 

\begin{itemize}
    \item[\textit{(ii)}] Let $|R|=r\leq t$. 
Note that $n\geq 56k^{2}s$ implies 
\[
|\hk(R)|=\binom{n-r}{k-r} \leq  \frac{k^{r-1}}{n^{r-1}} \binom{n-1}{k-1}\leq \frac{1}{28r(ks)^{r-1}}\binom{n-1}{k-1}.
\]
By \eqref{chernoff-large}, we infer that
\[
Pr\left(|\hk_p(R))| > \frac{1}{4r(ks)^{r-1}}p\binom{n-1}{k-1}\right)<e^{-\frac{1}{4r(ks)^{r-1}}p\binom{n-1}{k-1}}\le e^{-\frac{1}{4t(ks)^{t-1}}p\binom{n-1}{k-1}}. 
\]
Since the number of choices for $R$ is at most $\sum_{2\leq i\leq t}\binom{n}{i} \leq n^t$, by $p\geq \frac{8t^2(ks)^{t-1}\log n}{\binom{n-1}{k-1}}$ and the union bound, the probability that condition \textit{(ii)} is violated is at most
\[
n^t e^{-\frac{1}{4t(ks)^{t-1}}p\binom{n-1}{k-1}} \leq  e^{t\log n-\frac{1}{4t(ks)^{t-1}}p\binom{n-1}{k-1}}\le n^{-t}\rightarrow 0 \text{~as~}n\rightarrow\infty.
\]
Thus \textit{(ii)} holds.
    \item[\textit{(iii)}] Similarly,  $n\geq 56k^{2+1/t}s$ implies 
\[
|\hk(T)|=\binom{n-t-1}{k-t-1} \leq  \frac{k^t}{n^t} \binom{n-1}{k-1}\leq \frac{1}{28k^{t+1}s^t}\binom{n-1}{k-1}.
\]
By \eqref{chernoff-large}, we infer that
\[
Pr\left(|\hk_p(T))| > \frac{1}{4k^{t+1}s^t}p\binom{n-1}{k-1}\right)<e^{-\frac{1}{4k^{t+1}s^t}p\binom{n-1}{k-1}}. 
\]
Since the number of choices for $T$ is at most $\binom{n}{t+1} \leq n^{t+1}$, by $p\geq \frac{8(t+1)k^{t+1}s^t\log n}{\binom{n-1}{k-1}}$ and the union bound, the probability that condition \textit{(iii)} is violated is at most 
\[
n^{t+1} e^{-\frac{1}{4k^{t+1}s^t}p\binom{n-1}{k-1}} \leq  e^{(t+1)\log n-\frac{1}{4k^{t+1}s^t}p\binom{n-1}{k-1}}\le e^{(t+1)\log n-2(t+1)\log n}=n^{-t-1}\rightarrow 0 \text{~as~}n\rightarrow\infty.
\]
Thus \textit{(iii)} holds.
\end{itemize}

\end{proof}

\begin{proof}[Proof of Theorem \ref{thm:main-2}]
By Lemma \ref{lem-key2}, with high probability, $\hk_p$ satisfies \textit{(i)}, \textit{(ii)} and \textit{(iii)} of Lemma \ref{lem-key2}. Let us fix a hypergraph $\hf$  that satisfies \textit{(i)}, \textit{(ii)}, \textit{(iii)}  and let $\hh\subset \hf$ be a sub-hypergraph with $\nu(\hh)=s$ and  $|\hh|$ maximal. We have to show that $\hh$ is trivial.

Suppose indirectly that $\hh$ is non-trivial.  By Fact \ref{fact-3.1}, there exist $\ell \leq  \nu(\hh)$ and $T_1,T_2,\ldots,T_\ell$ such that \textit{(i)}, \textit{(ii)}, \textit{(iii)} of Fact \ref{fact-3.1} hold. Let $a_i=|T_i|$, $i=1,2,\ldots,\ell$. Recall that $\hh_0=\hh$ and $\hh_i=\hh-(T_1\cup \ldots\cup T_i)$, $i=1,2,\ldots,\ell$. Let $\hh_1'=\hh_0(\hat{T_1})$, $\hh_2'=\hh_1(\hat{T_2})$, $\ldots$, $\hh_\ell'=\hh_{\ell-1}(\hat{T_\ell})$ and $\hh_0'= \hh_\ell$. Since $\hh_0'$ is $t$-resilient, by Lemma \ref{lem-3.1} there exists $\hb_0$ of size at most $(k(s-\ell))^{t+1}$ such that $\hh_0'\subset \langle\hb_0\rangle$. Similarly, for $i\ge 1$, since $\hh_i'=\hh_{i-1}(\hat{T_i})$ and $\hh_{i-1}$ is $(a_i-1)$-resilient, by Lemma \ref{lem-3.1} there exists $\hb_i$ of size at most $a_i(k(s-i+1))^{a_i-1}$ such that $\hh_i'\subset \langle \hb_i\rangle$. Thus $\hh=\hh_0'\cup \hh_1'\cup \ldots\cup\hh_\ell'$ is contained in $\cup_{0\leq i\leq \ell}\langle\hb_i\rangle$.

By relabeling the indices if necessary, we may assume that $a_\ell=a_{\ell-1}=\ldots=a_{\ell-r+1}=1$ and $a_1\geq a_2\geq \ldots\geq a_{\ell-r}\geq 2$. Let $\hb_i=\{\{x_i\}\}$ for $i=\ell-r+1,\ldots,\ell$ and let $R=\{x_i\colon i=\ell-r+1,\ldots,\ell\}$. Let $q=s-r$. Note that for an arbitrary $Q\in \binom{[n]\setminus R}{q}$ the family $\hf(\bar{R},\hat{Q})$ can be used to replace $\hh(\bar{R})$ and the resulting family has matching number $s$.
To get a contradiction we need to show
\begin{align}
|\hf(\bar{R},\hat{Q})|>|\hh(\bar{R})|.
\end{align}

By Lemma \ref{lem-key2} \textit{(i)} we have  $|\hf(\bar{R},\hat{Q})| > \frac{1}{2}pq\binom{n-1}{k-1}$. By Lemma \ref{lem-key2} \textit{(ii)} and $|\hb_i|\leq a_i(k(s-i+1))^{a_i-1}$, we have
\[
\sum_{1\leq i\leq \ell-r}|\hh_i'| \leq \sum_{1\leq i\leq \ell-r}|\hf \cap \langle \hb_i \rangle| \leq \sum_{1\leq i\leq \ell-r} |\hb_i| \frac{1}{4a_i(ks)^{a_i-1}}p\binom{n-1}{k-1} \leq \frac{\ell-r}{4}p\binom{n-1}{k-1}.
\]
 By Lemma \ref{lem-key2} \textit{(iii)} and $|\hb_0|\leq (k(s-\ell))^{t+1}$, 
 \[
 |\hh_0'|\leq |\hf \cap \langle \hb_0 \rangle| \leq |\hb_0|  \frac{1}{4k^{t+1}s^t}p\binom{n-1}{k-1} \leq \frac{s-\ell}{4} p\binom{n-1}{k-1}. 
 \]
Thus,
\[
|\hh(\bar{R})|\leq \sum_{0\leq i\leq \ell-r}|\hh_i'|<\frac{s-\ell+\ell-r}{4} p\binom{n-1}{k-1}=\frac{q}{4} p\binom{n-1}{k-1}<|\hf(\bar{R},\hat{Q})|,
\]
as desired. Thus the theorem holds. 
\end{proof}

\begin{proof}[Proof of Proposition \ref{prop-2.3}]
Let $N$ denote the number of $(s+1)$-matchings in $\binom{[n]}{k}$. Then
\[
N=\binom{n}{(s+1)k} \frac{((s+1)k)!}{k!^{s+1}}.
\]
Since the probability that a fixed $(s+1)$-matching is in $\hk_p$ is $p^{s+1}$, by the union bound we have
\[
Pr(\hk_p \mbox{ contains an $(s+1)$-matching}) \leq Np^{s+1} \leq \binom{n}{(s+1)k} \frac{((s+1)k)!}{k!^{s+1}}p^{s+1}.
\]
We use $(n)_k$ to denote $n(n-1)\ldots(n-k+1)$.
Note that
\[
\binom{n}{(s+1)k} \frac{((s+1)k)!} {(k!)^{s+1}}=  \frac{n(n-1)\ldots(n-(s+1)k+1)}{(k!)^{s+1}} \leq \frac{(n)_k(n-k)_k\ldots(n-sk)_k}{(k!)^{s+1}}
\]
and
\[
(n-ik)_k \leq \left(1-\frac{ik}{n}\right)^k (n)_k \leq e^{-ik^2/n} (n)_k.
\]
Using $p\ll e^{\frac{k^2s}{2n}}/\binom{n}{k}$, we obtain that
\[
\frac{(n)_k(n-k)_k\ldots(n-sk)_k}{(k!)^{s+1}}p^{s+1}  \leq e^{-(1+2+\ldots+s)k^2/n}\binom{n}{k}^{s+1}p^{s+1} \leq  e^{-s(s+1)k^2/(2n)} \binom{n}{k}^{s+1}p^{s+1}\ll 1.
\]
On the other hand, by the union bound and $p\gg \frac{s\log n}{\binom{n-s}{k}}$ we have
\[
Pr(\hk_p \mbox{ is trivial}) \leq \binom{n}{s} (1-p)^{\binom{n-s}{k}}< \binom{n}{s} \exp\left(-p\binom{n-s}{k}\right)\ll 1.
\]
Thus with high probability, $\hk_p$ has matching number at most $s$ and is non-trivial.
\end{proof}

\section{Proof of Theorem \ref{thm:main-3}}
In this section, we prove Theorem \ref{thm:main-3}, which is a result specifically for graphs.

Let $(B,A_1,A_2,\ldots,A_m)$ be a partition of $[n]$. We say that $(B,A_1,A_2,\ldots,A_m)$ is an $s$-partition if the following hold:
\begin{itemize}
    \item[(i)] Let $|A_i|=a_i$, $i=1,2,\ldots,m$, and let $|B|=b$. Each $a_i$ is odd and $a_1\geq a_2\geq \ldots\geq a_m\geq 1$;
    \item[(ii)] $b+\sum\limits_{1\leq i\leq m} \frac{a_i-1}{2}=s$.
\end{itemize}
For an $s$-partition $(B,A_1,A_2,\ldots,A_m)$ of $[n]$, define a graph $G(B,A_1,A_2,\ldots,A_m)$ on the vertex set $[n]$ so that $G[B]$ and $G[A_i]$, $i=1,2,\ldots,m$, are all cliques, and $G[B,A_1\cup \ldots\cup A_m]$ is complete bipartite.

\begin{thm}[Tutte-Berge theorem  or the Edmonds-Gallai Theorem, see \cite{lovasz2009matching}]\label{thm-tbeg}
Let $G$ be a graph on the vertex set $[n]$ with $\nu(G)\leq s$ and $n\geq 2s+2$. Then there is an $s$-partition $(B,A_1,A_2,\ldots,A_m)$ of $[n]$ such that $G$ is a subgraph of $G(B,A_1,A_2,\ldots,A_m)$.
\end{thm}

By applying Theorem \ref{thm-tbeg}, Erd\H{o}s and Gallai determined the maximum number of edges in a graph with matching number at most $s$, which is the $k=2$ case of the Erd\H{o}s Matching Conjecture.

\begin{thm}[\cite{erd6s1959maximal}]
Let $G$ be a graph on $n$ vertices. If $\nu(G)\leq s$ and $n\geq 2s+2$, then
\[
e(G) \leq \max\left\{\binom{2s+1}{2},\binom{s}{2}+s(n-s) \right\}.
\]
\end{thm}

Let us mention that very recently by combining the Tutte-Berge theorem and other techniques, Alon and the first author \cite{alon2024turan} determined the maximum number of edges in a $K_{r+1}$-free graph with matching number at most $s$.

Note that
\[
\binom{2r_1+1}{2} +\binom{2r_2+1}{2} \leq \binom{2(r_1+r_2)+1}{2}.
\]
Let $a_i=2r_i+1$, $i=1,2,\ldots,m$ and let $r_\ell\geq 1$ and $r_{\ell+1}=\ldots=r_m=0$. By \textit{(ii)} we have
$r_1+r_2+\ldots+r_\ell+b=s$.
Then
\begin{align}
e(G(B,A_1,A_2,\ldots,A_m)) &= \binom{b}{2} +\sum_{1\leq i\leq \ell}\binom{a_i}{2} +b\sum_{1\leq i\leq m} a_i\nonumber\\[3pt]
&\leq \binom{b}{2} +\binom{2(r_1+r_2+\ldots+r_\ell)+1}{2} +b(n-b)\nonumber\\[3pt]
&\leq \binom{b}{2} +\binom{2(s-b)+1}{2} +b(n-b)\nonumber\\[3pt]
&\leq \max\left\{\binom{2s+1}{2},\binom{s}{2}+s(n-s)\right\}.
\end{align}

\begin{proof}[Proof of Theorem\ref{thm:main-3}]
Let
\[
f(n,s):=\max\left\{\binom{2s+1}{2},\binom{s}{2}+s(n-s)\right\}.
\]
Let $\hg$ be the family of all graphs with $\nu(G)\leq s$ on the vertex set $[n]$. Then
\begin{align*}
&Pr\left(X\geq (1+\varepsilon)pf(n,s)\right)\\[3pt]
=&Pr\left(\exists G\in \hg, e(G\cap G(n,p))\geq (1+\varepsilon)pf(n,s)\right)\\[3pt]
=&Pr\left(\exists \mbox{ an $s$-parition }(B,A_1,A_2,\ldots,A_m), e(G(B,A_1,A_2,\ldots,A_m)\cap G(n,p))\geq (1+\varepsilon)pf(n,s)\right).
\end{align*}

Assume that  $|A_{\ell+1}|=\cdots =|A_m|=1$ for some $\ell\in [m]$, the $s$-partition $(B,A_1,A_2,\ldots,A_m)$ is determined by $(B,A_1,A_2,\ldots,A_\ell)$. By \textit{(ii)} we have
\[
b+\sum\limits_{1\leq i\leq \ell} \frac{a_i-1}{2}=s.
\]
It follows that $\ell\leq s-1$ and $b+a_1+a_2+\ldots+a_\ell< 3s$.  Thus the total number of $s$-partitions is at most $s^{3s}n^{3s}$.

Let $(B,A_1,A_2,\ldots,A_m)$ be an  $s$-partition of $[n]$ and let  $G=G(B,A_1,A_2,\ldots,A_m)$. If $e(G)\geq \frac{1}{7}f(n,s)$, then by \eqref{chernoff-small},
\[
Pr\left(e(G\cap G(n,p))\geq (1+\varepsilon)pf(n,s)>(1+\varepsilon)pe(G)\right)\leq e^{-\frac{\varepsilon^2pe(G)}{3}} \leq e^{-\frac{\varepsilon^2pf(n,s)}{21}}.
\]
If $e(G)< \frac{1}{7}f(n,s)$, then $(1+\varepsilon)pf(n,s)\geq 7pe(G)$. By \eqref{chernoff-large} we have
\[
Pr\left(e(G\cap G(n,p))\geq (1+\varepsilon)pf(n,s)\right) \leq e^{-(1+\varepsilon)pf(n,s)}\leq e^{-\frac{\varepsilon^2pf(n,s)}{21}}.
\]
Thus, by the union bound we conclude that
\[
Pr\left(\exists G\in \hg, e(G\cap G(n,p))\geq (1+\varepsilon)pf(n,s)\right)\leq s^{3s}n^{3s}e^{-\frac{\varepsilon^2pf(n,s)}{21}} =e^{-\frac{11f(n,s)\log n}{n}+3s\log(sn)}.
\]
If $n\geq \frac{5s+3}{2}$, then $s<\frac{2n}{5}$ and $f(n,s)=\binom{s}{2}+s(n-s)> \frac{3sn}{4}$. If $2s+2 \leq n< \frac{5s+3}{2}$, then $f(n,s)=\binom{2s+1}{2}>\frac{3sn}{4}$ as well. It follows that
\[
\frac{11f(n,s)\log n}{n}-3s\log(sn) > 7s\log n - 6s\log n =s\log n.
\]
Thus with probability at least $1-n^{-s}$,
\[
X\leq (1+\varepsilon)p\max\left\{\binom{2s+1}{2},\binom{s}{2}+s(n-s)\right\}.
\]

For the lower bound, let us consider the graph $G_1=\binom{[2s+1]}{2}$ and $G_2=\binom{[s]}{2}\cup ([s]\times [s+1,n])=\binom{[n]}{2}\setminus \binom{[s+1,n]}{2}$. Then by \eqref{chernoff-small} we have
\[
Pr\left(e(G_1\cap G(n,p))\leq (1-\varepsilon)p e(G_1)\right)\leq  e^{-\frac{\varepsilon^2pe(G_1)}{3}}.
\]
and
\[
Pr\left(e(G_2\cap G(n,p))\leq (1-\varepsilon)p e(G_2)\right)\leq  e^{-\frac{\varepsilon^2pe(G_2)}{3}}.
\]
Since $\max\{e(G_1),e(G_2)\}\geq \frac{3sn}{4}$, we infer that for $i=1$ or $2$,
\[
\frac{\varepsilon^2pe(G_i)}{3} \geq  \frac{250\log n}{3n} \times  \frac{3sn}{4}\geq 60s \log n.
\]
It follows that with probability $1-n^{-60s}$, we have
\[
X \geq (1-\varepsilon)p\max\left\{\binom{2s+1}{2},\binom{s}{2}+s(n-s)\right\}.
\]
Thus the theorem is proven.
\end{proof}

\bibliographystyle{abbrv}
\bibliography{refs}

\begin{thebibliography}{10}

\bibitem{alon2024turan}
N.~Alon and P.~Frankl.
\newblock Tur{\'a}n graphs with bounded matching number.
\newblock {\em Journal of Combinatorial Theory, Series B}, 165:223--229, 2024.

\bibitem{balogh2009erdHos}
J.~Balogh, T.~Bohman, and D.~Mubayi.
\newblock Erd{\H{o}}s--{Ko}--{Rado} in random hypergraphs.
\newblock {\em Combinatorics, Probability and Computing}, 18(5):629--646, 2009.

\bibitem{balogh2023sharp}
J.~Balogh, R.~A. Krueger, and H.~Luo.
\newblock Sharp threshold for the {E}rd{\H{o}}s--{K}o--{R}ado theorem.
\newblock {\em Random Structures \& Algorithms}, 62(1):3--28, 2023.

\bibitem{BDE}
B.~Bollob\'{a}s, D.~E. Daykin, and P.~Erd\H{o}s.
\newblock {Sets of independent edges of a hypergraph}.
\newblock {\em The Quarterly Journal of Mathematics}, 27(1):25--32, 03 1976.

\bibitem{erd6s1959maximal}
P.~Erd\H{o}s and T.~Gallai.
\newblock On maximal paths and circuits of graphs.
\newblock {\em Acta Math. Acad. Sci. Hung}, 10, 1959.

\bibitem{erdos1965problem}
P.~Erd{\H{o}}s.
\newblock A problem on independent $r$-tuples.
\newblock {\em Ann. Univ. Sci. Budapest. E{\"o}tv{\"o}s Sect. Math}, 8(93-95):2, 1965.

\bibitem{erdos1961intersection}
P.~Erd{\H{o}}s, C.~Ko, and R.~Rado.
\newblock Intersection theorems for systems of finite sets.
\newblock {\em Quart. J. Math. Oxford Ser.(2)}, 12:313--320, 1961.

\bibitem{frankl2013improved}
P.~Frankl.
\newblock Improved bounds for {Erd{\H{o}}s} matching conjecture.
\newblock {\em Journal of Combinatorial Theory, Series A}, 120(5):1068--1072, 2013.

\bibitem{frankl2018resilient}
P.~Frankl.
\newblock Resilient hypergraphs with fixed matching number.
\newblock {\em Combinatorica}, 38:1079--1094, 2018.

\bibitem{frankl2022erdHos}
P.~Frankl and A.~Kupavskii.
\newblock The {Erd{\H{o}}s} matching conjecture and concentration inequalities.
\newblock {\em Journal of Combinatorial Theory, Series B}, 157:366--400, 2022.

\bibitem{gauy2017erdHos}
M.~M. Gauy, H.~Han, and I.~C. Oliveira.
\newblock Erd{\H{o}}s--{Ko}--{Rado} for random hypergraphs: asymptotics and stability.
\newblock {\em Combinatorics, Probability and Computing}, 26(3):406--422, 2017.

\bibitem{hamm2019erdHosI}
A.~Hamm and J.~Kahn.
\newblock On {Erd{\H{o}}s}--{Ko}--{Rado} for random hypergraphs {I}.
\newblock {\em Combinatorics, Probability and Computing}, 28(6):881--916, 2019.

\bibitem{hamm2019erdHos}
A.~Hamm and J.~Kahn.
\newblock On {Erd{\H{o}}s}--{Ko}--{Rado} for random hypergraphs {II}.
\newblock {\em Combinatorics, Probability and Computing}, 28(1):61--80, 2019.

\bibitem{janson2011random}
S.~Janson, T.~Luczak, and A.~Rucinski.
\newblock {\em Random graphs}.
\newblock John Wiley \& Sons, 2011.

\bibitem{lovasz2009matching}
L.~Lov{\'a}sz and M.~D. Plummer.
\newblock {\em Matching theory}, volume 367.
\newblock American Mathematical Soc., 2009.

\end{thebibliography}

\end{document}